\documentclass[11pt]{amsart}
\usepackage[utf8]{inputenc}
\usepackage[english]{babel}
\usepackage{csquotes, xpatch}
\usepackage{amsmath}
\usepackage{amssymb}
\usepackage{amsthm}
\usepackage{hyperref}
\usepackage{mathtools}
\usepackage{xcolor}
\usepackage{enumerate}
\usepackage{tikz}
\usetikzlibrary{arrows.meta}

\usepackage{geometry}[a4paper]
\usepackage[
  backend=biber,
  style=numeric,
  sorting=nyt,
  sortcites=true,
  autopunct=true,
  maxnames=9
]{biblatex}
\theoremstyle{plain}
\newtheorem{theorem}{Theorem}[section]
\newtheorem{lemma}[theorem]{Lemma}
\newtheorem{proposition}[theorem]{Proposition}

\theoremstyle{definition}
\newtheorem{definition}[theorem]{Definition}

\newtheorem{remark}[theorem]{Remark}

\newcommand{\KP}{\mathbb{KP}}
\newcommand{\op}[1]{#1^{\operatorname{op}}}
\newcommand{\CC}{\mathbb{C}}
\newcommand{\RR}{\mathbb{R}}
\newcommand{\NN}{\mathbb{N}}
\newcommand{\ZZ}{\mathbb{Z}}
\newcommand{\ie}{i.e.~}

\newcommand{\Fourier}{\mathcal{F}}
\newcommand{\Ad}{\operatorname{Ad}}
\newcommand{\sgn}{\operatorname{sign}}
\newcommand{\Tr}{\operatorname{Tr}}
\newcommand{\Cay}{\operatorname{Cay}}

\DeclareMathOperator{\rank}{rank}
\DeclareMathOperator{\linearspan}{span}
\DeclareMathOperator{\id}{\operatorname{id}}

\title{The Quantum Cayley Graphs of the Kac-Paljutkin Quantum Group}
\author[Junk]{Luca Junk}
\date{\today}
\address{Universität des Saarlandes, Fachbereich Mathematik, Saarbrücken, Germany}
\email{\href{mailto:junk@math.uni-sb.de}{junk@math.uni-sb.de}}

\begin{document}
\allowdisplaybreaks

\begin{abstract}
  We study the quantum Cayley graphs associated with the Kac-Paljutkin quantum group. We determine which projections give rise to undirected, irreflexive and connected quantum graphs and classify them up to isomorphism. Their spectra and chromatic numbers are computed. For the special case of the projection \(p = I_2\) we also compute the associated quantum Cuntz-Krieger algebra.
\end{abstract}

\maketitle

\section{Introduction}

The notion of quantum (or non-commutative) graphs was introduced in the context of quantum information theory in \cite{DuanSeveriniWinter2011} where they were studied as the confusability graphs of quantum channels. Later Weaver \cite{Weaver2012,Weaver2021} embedded them into his framework of quantum relations using operator systems and Musto, Reutter and Verdon \cite{MustoReutterVerdon2018} found a description using quantum adjacency operators and edge projections. The equivalence of these different viewpoints was shown in \cite{MustoReutterVerdon2018} for the tracial setting and in \cite{Daws2024} and in \cite{Wasilewski2024} in the general case. Since their introduction, quantum graphs have developed into an interesting theory generalizing classical graphs and with links to operator algebras \cite{BrannanEiflerVoigtWeber2022, BrannanHamidiIsmertNelsonWasilewski2023, HamidiIsmertNelson2026}, quantum (automorphism) groups \cite{BrannanCulfVernooij2025, BrannanChirvasituEiflerHarrisPaulsenSuWasilewski2019, Wasilewski2024, Daws2024} and quantum complexity \cite{CulfMehta2025} among others.

In constrast to classical graphs, it is a non-trivial problem to find interesting (\ie connected, regular, etc.) explicit examples of quantum graphs. An important step in providing more examples was taken with the classification of quantum graphs on \(M_2(\CC)\) \cite{Gromada2022, Matsuda2022, KieferSchaefer2026}. This article aims to contribute in a similar style by classifying all (undirected, irreflexive and connected) quantum graphs that appear as the quantum Cayley graphs of the eight-dimensional Kac-Paljutkin quantum group \(\KP\) which was introduced in \cite{KacPaljutkin1966}. The framework of quantum Cayley graphs --- developed by Wasilewski in \cite{Wasilewski2024} --- is particularly suitable since it is easy to produce regular and connected quantum graphs with this. The choice of the Kac-Paljutkin quantum group is also natural, since it is the smallest quantum group (in the sense of Woronowicz \cite{Woronowicz1987, Woronowicz1998}) which is neither commutative nor co-commutative, \ie sufficiently non-classical to produce interesting new quantum graphs but still small enough to make a complete classification feasible.

\vspace*{1em}

Our results include a full classification of the generating projections in \(C(\KP) = \CC^4 \oplus M_2(\CC)\) which give rise to undirected, irreflexive and connected quantum graphs (Proposition~\ref{prop:projection_classification}). The resulting quantum graphs are classified up to isomorphism (Theorem~\ref{thm:isomorphism_classification}) and their spectra (Proposition~\ref{prop:spectrum}) and chromatic numbers (Theorem~\ref{thm:chromatic_numbers}) are computed. We show that the specific generating projection \(p= I_2 = 1_{M_2(\CC)} \in C(\KP)\) yields a quantum graph which is quantum isomorphic to the classical complete bipartite graph \(K_{4,4}\) (Proposition~\ref{prop:qisoK44}) and we determine the associated quantum Cuntz-Krieger algebra (as defined in \cite{BrannanEiflerVoigtWeber2022}) to be isomorphic to \(M_5(\CC) \otimes \mathcal{O}_{16}\)  (Theorem~\ref{thm:qCuntzKrieger}).

\vspace*{1em}

The paper is organized as follows. In Section 2 we give a brief overview over the general theory of quantum graphs, define their chromatic numbers and explain the quantum Cayley graph construction of \cite{Wasilewski2024}. In Section 3 we recall the definition of the Kac-Paljutkin quantum group, characterize the projections which yield undirected, irreflexive and connected quantum Cayley graphs and determine their isomorphism classes. Moreover, we compute the spectrum and the chromatic numbers of the resulting quantum graphs. In Section 4 we focus on the generating projection \(p = I_2\), showing that the this gives a quantum graph quantum isomorphic to the complete bipartite graph \(K_{4,4}\) and compute the corresponding quantum Cuntz-Krieger algebra.

\section{Preliminaries}

\subsection*{Quantum Graphs}\hfill\\
We briefly review the notion of quantum graphs and introduce the concepts that are needed for this article. The general references for this are \cite{Daws2024,Matsuda2022,Matsuda2024,Wasilewski2024}.

Quantum graphs generalize finite graphs in the sense of non-commutative topology. The role of the finite vertex set is played by a finite-dimensional \(C^*\)-algebra \(B\), equipped with a faithful linear functional \(\psi \colon B \to \CC\), which should be thought of as the analogue of the counting measure on the set of vertices.

We denote the GNS Hilbert space associated to \((B,\psi)\) by \(L^2(B,\psi)\), or simply by \(L^2(B)\). Since \(\psi\) is faithful and \(B\) is finite-dimensional, the GNS map is a linear isomorphism. We therefore identify \(B\) and \(L^2(B)\) as linear spaces and suppress the GNS map from the notation. Thus an element \(x \in B\) will also be regarded as a vector in \(L^2(B)\). We use the convention that inner products are anti-linear in the first variable, so that
\[
  \langle x,y\rangle = \psi(x^*y), \qquad x,y \in B.
\]
Let \(m \colon B \otimes B \to B\) denote the multiplication map. Via the above identification, we also regard \(m\) as a linear map
\[
  m \colon L^2(B) \otimes L^2(B) \to L^2(B),
\]
and we write
\[
  m^* \colon L^2(B) \to L^2(B) \otimes L^2(B)
\]
for its Hilbert-space adjoint.

In the literature on quantum graphs one often assumes that \(\psi\) is a \emph{\(\delta\)-form}, that is, a faithful state satisfying
\[
  mm^* = \delta^2 \id
\]
for some \(\delta>0\). We instead follow the normalization used in \cite{Wasilewski2024} and assume throughout that
\[
  mm^* = \id.
\]
A characterization of all faithful functionals with this property can be obtained from \autocite[Proposition~2.3]{Wasilewski2024}.

A graph structure on a finite quantum space \((B,\psi)\) can be encoded in three equivalent ways: by a quantum adjacency operator, by a quantum edge projection, or by a quantum relation. The last of these should be viewed as an operator-bimodule analogue of the space of functions on the edge set of a classical graph. We recall all three descriptions below and refer to the literature for the explicit equivalences between them.

For the rest of this subsection, we fix a finite-dimensional \(C^*\)-algebra \(B\) together with a faithful functional \(\psi \colon B \to \CC\) satisfying \(mm^*=\id\).

\begin{definition}
  A \emph{quantum adjacency matrix} (or \emph{quantum adjacency operator}) on \((B, \psi)\) is a completely positive linear map \(A \colon B \to B\) satisfying
  \begin{equation}
    m (A \otimes A) m^* = A. \label{eq:schur_idempotent}
  \end{equation}
\end{definition}

Let us compare this condition with the classical case. If \(G\) is a finite graph with vertex set \(V\), then the corresponding finite quantum space is \(B=C(V)\), equipped with the counting functional
\[
  \psi(f)=\sum_{v \in V} f(v).
\]
Writing \((\delta_v)_{v \in V}\) for the canonical basis of \(C(V)\), where \(\delta_v\) is the indicator function of the vertex \(v\), one has
\[
  m^*(\delta_v)=\delta_v \otimes \delta_v.
\]
In particular, \(mm^*=\id\). For two linear maps \(S,T \colon C(V) \to C(V)\), we obtain
\begin{align*}
  m(S \otimes T)m^*(\delta_v) &= S(\delta_v) T(\delta_v) = \sum_{u,w \in V} S_{vu} \delta_u \cdot T_{vw} \delta_w\\
  &= \sum_{w \in V} S_{vw} T_{vw} \delta_w = \sum_{w \in V} (S \bullet T)_{vw} \delta_w\\
  &= (S \bullet T)(\delta_v)
\end{align*}
where \(\bullet\) denotes entrywise multiplication (also called Schur-- or Hadamard--product). Thus the formula
\[
  S \bullet T \coloneqq m (S \otimes T) m^*, \qquad (S, T \colon B \to B)
\]
is the non-commutative analogue of the Schur product. Consequently, condition \eqref{eq:schur_idempotent} says precisely that \(A\) is idempotent with respect to this Schur product. In the classical case this is equivalent to saying that the entries of the adjacency matrix are \(0\) or \(1\).

We will use the following observation repeatedly without further comment. Under the Schur-idempotence condition \eqref{eq:schur_idempotent}, complete positivity of \(A\) is equivalent to \(A\) being \(*\)-pre\-ser\-ving, or \emph{real} in the terminology of \cite{Matsuda2022} (see \autocite[Proposition~2.23]{Matsuda2022}).

\begin{definition}
  A \emph{quantum edge projection} on \(B\) is a projection \(p \in B \otimes \op{B}\) where \(\op{B}\) denotes the opposite algebra of \(B\).
\end{definition}

For a classical graph with vertex set \(V\) and edge set \(E \subseteq V \times V\), one has
\[
  C(V) \otimes \op{C(V)} \cong C(V \times V),
\]
and the corresponding edge projection is the indicator function of \(E\).

\begin{definition}
  Let \(B\) be faithfully represented on a Hilbert space \(H\), so we can assume \(B \subseteq B(H)\). A \emph{quantum relation} on \(B\) is a weak-\(*\)-closed \(B'\)-\(B'\)-bimodule \(S \subseteq B(H)\).
\end{definition}

The three descriptions above are equivalent: quantum adjacency operators, quantum edge projections, and quantum relations are in natural one-to-one correspondence. For proofs and explicit descriptions of the bijections, see for example \cite{Daws2024}, \cite{MustoReutterVerdon2018}, and \cite{Wasilewski2024}. This justifies the following definition.

\begin{definition}
  A \emph{quantum graph} \(G\) consists of a finite quantum space \((B, \psi)\), together with any one of the following equivalent pieces of data:
  \begin{itemize}
    \item a quantum adjacency operator \(A \colon B \to B\),
    \item a quantum edge projection \(p_{\text{edge}} \in B \otimes \op{B}\),
    \item a quantum relation on \(B\), relative to a faithful representation \(B \subseteq B(H)\).
  \end{itemize}
  We say that \(G\) is
  \begin{itemize}
    \item \emph{tracial} if \(\psi\) is tracial,
    \item \emph{reflexive}, or \emph{has all loops}, if \(A \bullet \id = \id\),
    \item \emph{irreflexive}, or \emph{has no loops}, if \(A \bullet \id = 0\),
    \item \emph{(GNS-)undirected} if \(A\) is self-adjoint with respect to the GNS inner product,
    \item \emph{\(d\)-regular} if \(A(1) = d \cdot 1 = A^*(1)\) for some \(d > 0\),
    \item \emph{connected} if \(G\) is undirected and there is no non-trivial projection \(p \in B \setminus \{0,1\}\) such that \(A\) commutes with the left-regular representation \(\lambda(p)\) of \(p\) (see \cite{CourtneyGanesanWasilewski2025} for several equivalent characterizations of connectedness).
  \end{itemize}
\end{definition}

In this article we will almost exclusively be concerned with the adjacency operator picture of quantum graphs.

In the non-tracial setting, one may also use the KMS inner product
\[
  \langle x,y\rangle_{\mathrm{KMS}}
  \coloneqq
  \psi\bigl(x^*\sigma_{-i/2}(y)\bigr),
  \qquad x,y \in B,
\]
where \((\sigma_{z})_{z \in \CC}\) denotes the (analytic continuation of the) modular automorphism group associated to \((B,\psi)\). As explained in \cite{Wasilewski2024}, the KMS inner product is better behaved with respect to positivity, and GNS self-adjointness can be too restrictive as a definition of undirectedness. Since the quantum graphs considered in this article are tracial, the GNS and KMS inner products coincide. We therefore do not distinguish between these notions in what follows.

\begin{definition}
  Two quantum graphs \(G = (B, \psi, A)\) and \(G' = (B', \psi', A')\) are \emph{isomorphic} if there is a unital \(*\)-isomorphism \(\theta \colon B \to B'\) such that
  \[
    \psi' \circ \theta = \psi \quad \text{and} \quad \theta A = A' \theta.
  \]
\end{definition}

\begin{definition}
  Let \(G = (B, \psi, A)\) be an irreflexive quantum graph and let \(S \subseteq B(H)\) be the associated operator bimodule. A \emph{\(k\)-coloring} of \(G\) is given by projections \(P_1, \ldots, P_k \in B \otimes M\) for some finite von Neumann algebra \(M\) with a faithful normal trace such that
  \[
    \sum_{i = 1}^k P_i = 1_{B \otimes M}
  \]
  and
  \[
    P_i (X \otimes 1_M) P_i = 0, \qquad 1 \leq i \leq k,
  \]
  for every \(X \in S\). The coloring is called \emph{classical} (or \emph{local}) if \(\dim(M) = 1\), \emph{quantum} if \(\dim(M) < \infty\) and \emph{quantum commuting} if \(M\) is infinite-dimensional. The classical, quantum and quantum commuting \emph{chromatic numbers} of \(G\) are then defined as the smallest \(k \in \NN\) such that there exists a classical/quantum/quantum commuting \(k\)-coloring of \(G\) (or \(\infty\) if no finite coloring exists) and denoted by \(\chi(G)\), \(\chi_q(G)\) and \(\chi_{qc}(G)\) respectively.
\end{definition}

The next lemma characterizes classical colorings in the tracial case purely in terms of the quantum adjacency matrix without reference to the operator bimodule. This is probably known to experts, but we could not find this specific formulation in the literature, so we give a short proof.

\begin{lemma}\label{lem:coloring_equivalence}
  Let \(G = (B, \psi, A)\) be an irreflexive tracial quantum graph. A family of projections \(P_1, \ldots, P_k \in B\) such that \(\sum_{i=1}^k P_i = 1\) is a classical \(k\)-coloring of \(G\) if and only if
  \[
    P_i A(P_i) = 0
  \]
  for all \(1 \leq i \leq k\).
\end{lemma}
\begin{proof}
  As a finite-dimensional \(C^*\)-algebra, \(B\) is of the form \(B = \bigoplus_{\alpha = 1}^d M_{n_\alpha}(\CC)\) and the condition \(mm^* = \id\) together with traciality force \(\psi\) to be of the form \(\psi = \bigoplus_{\alpha = 1}^d n_\alpha \Tr_{n_\alpha}\) by \autocite[Proposition~2.3]{Wasilewski2024}. Since the quantum relation associated to \(G\) is essentially independent of the faithful representation of \(B\) by \autocite[Theorem~2.7]{Weaver2012} we may assume that \(B\) is represented faithfully on \(H = \bigoplus_{\alpha=1}^d \CC^{n_\alpha}\) in the natural way. The quantum relation \(S\) then decomposes into blocks \(S_{\alpha \beta} \subseteq B(\CC^{n_\alpha}, \CC^{n_\beta})\) and we can choose an orthonormal basis \((S_{\alpha \beta}^i)_i\) for each block such that
  \begin{equation}\label{eq:Kraus_form_ONB}
    A(x) = \bigoplus_{\beta=1}^d \sum_{\alpha=1}^d \sum_i \sqrt{\frac{n_\alpha}{n_\beta}} S_{\alpha \beta}^i x_\alpha (S_{\alpha \beta}^i)^*, \qquad x \in B,
  \end{equation}
  by \autocite[Proposition~3.30]{Wasilewski2024}.

  Now let \(P = \bigoplus_{\alpha=1}^d P_\alpha \in B\) be any projection. It suffices to show that \(P X P = 0\) for all \(X \in S\) is equivalent to \(P A(P) = 0\). The \(\beta\) block of \(PA(P)P\) is, by (\ref{eq:Kraus_form_ONB}),
  \begin{align*}
    (PA(P)P)_\beta &= P_\beta \Big( \sum_\alpha \sum_i \sqrt{\frac{n_\alpha}{n_\beta}} S_{\alpha \beta}^i P_\alpha (S_{\alpha \beta}^i)^* \Big) P_\beta\\
    &= \sum_\alpha \sum_i \sqrt{\frac{n_\alpha}{n_\beta}} (P_\beta S_{\alpha \beta}^i P_\alpha) (P_\beta S_{\alpha \beta}^i P_\alpha)^*.
  \end{align*}
  Every summand above is positive, so \(PA(P)P = 0\) if and only if \(P_\beta S_{\alpha \beta}^i P_\alpha = 0\) for every \(\alpha, \beta, i\). Since \((S_{\alpha \beta}^i)_i\) spans \(S_{\alpha \beta}\), this is equivalent to \(P X P = 0\) for every \(X \in S\). Moreover, by the (complete) positivity of \(A\), \(P A(P) P = 0\) is equivalent to \(P A(P) = 0\). This proves the claim.
\end{proof}

\subsection*{Finite Quantum Groups}\hfill\\
Finite quantum groups are the non-commutative analogues of finite groups. In the same way that a finite quantum space is described by a finite-dimensional \(C^*\)-algebra, a finite quantum group is described by a finite-dimensional \(C^*\)-algebra equipped with a coproduct encoding the group multiplication. They form a natural subclass of the compact quantum groups defined by Woronowicz in \cite{Woronowicz1998}. For a specific treatment of finite quantum groups we refer to \cite{VanDaele1997}.

\begin{definition}
  A \emph{finite quantum group} is a pair \(G=(A,\Delta)\), where \(A\) is a finite-dimensional \(C^*\)-algebra and \(\Delta \colon A \to A \otimes A\) is a unital \(*\)-homomorphism satisfying the coassociativity condition
  \[
    (\Delta \otimes \id)\Delta = (\id \otimes \Delta)\Delta.
  \]
  Moreover, we require the cancellation conditions
  \[
    \linearspan \{\Delta(a)(1 \otimes b) : a,b \in A\} = A \otimes A = \linearspan \{\Delta(a)(b \otimes 1) : a,b \in A\}.
  \]
\end{definition}

Let \(\Gamma\) be a finite group. Then it becomes a finite quantum group if we take \(A = C(\Gamma)\) and \(\Delta\) to be the dual of the group law, \ie
\begin{align*}
  &\Delta \colon C(\Gamma) \to C(\Gamma) \otimes C(\Gamma) \cong C(\Gamma \times \Gamma)\\
  &\Delta(f)(s,t) = f(st), \qquad f \in C(\Gamma), \ s,t \in \Gamma.
\end{align*}
The coassociativity of \(\Delta\) is precisely the associativity of the multiplication on \(\Gamma\), while the cancellation conditions correspond to the usual cancellation property in the group. Moreover, all finite quantum groups with commutative algebra \(A\) are of this form.

For two functionals \(\eta, \omega \colon C(G) \to \CC\) we can define their convolution product \(\eta * \omega \colon C(G) \to \CC\) as
\[
  \eta * \omega \coloneqq (\eta \otimes \omega) \Delta.
\]
Every finite quantum group \(G\) admits a unique Haar state, \ie a state \(h \colon C(G) \to \CC\) such that
\[
  \omega * h = h * \omega = \omega(1) h
\]
for all \(\omega \in C(G)^*\). The Haar state is in fact faithful (see \cite{VanDaele1997}) and in the case of a classical group \(\Gamma\), it is given by integration against the normalized counting measure.

In order to match the conventions of \cite{Wasilewski2024} we will however rescale the Haar state and work instead with the functional \(H \coloneqq \dim C(G) \cdot h\).

The Fourier transform \(\Fourier \colon C(G) \to C(G)^*\), which is given by
\[
  \Fourier(a)(b) \coloneqq H(ab), \qquad a,b \in C(G),
\]
allows us to also define a convolution product for elements of \(C(G)\) by setting
\[
  a * b \coloneqq \Fourier^{-1}( \Fourier(a) * \Fourier(b) ), \qquad a,b \in C(G).
\]

Every finite quantum group carries a Hopf algebra structure, \ie there exists a counit \(\varepsilon \colon C(G) \to \CC\) and an antipode \(S \colon C(G) \to C(G)\) which --- together with the comultiplication \(\Delta\) --- satisfy the usual Hopf algebra axioms. See \autocite[Theorem~5.4.1]{Timmermann2008} for a proof.

\subsection*{Quantum Cayley Graphs}\hfill\\
In \cite{Wasilewski2024} the usual Cayley graph construction was generalized to discrete quantum groups and quantum graphs. We will describe it briefly in the --- far simpler --- case of finite quantum groups where the left and right Haar weights coincide and are in fact bounded linear functionals.

Let \(G\) be a finite quantum group and denote the multiplication map of \(C(G)\) by \(m\). From the discussion in \autocite[Section~4.1]{Wasilewski2024} we have that \(mm^* = \id\) if we equip \(C(G)\) with the Haar functional \(H\). For an element \(x \in C(G)\) we can consider its associated left convolution operator \(T_x \colon C(G) \to C(G)\), \(T_x(a) \coloneqq x * a\). Then we have \(T_x \bullet T_y = T_{xy}\) for \(x, y \in C(G)\) (\autocite[Proposition~4.8]{Wasilewski2024}), so \(T_x\) is Schur idempotent if and only if \(x\) is an idempotent and it is \(*\)-preserving if and only if \(x\) is self-adjoint (\autocite[Proposition~4.9]{Wasilewski2024}). Moreover, if \(x = x^*\), then \(T_x\) is GNS self-adjoint if and only if \(x = S(x)\) and KMS self-adjoint if and only if \(x = R(x)\) where \(R\) is the unitary antipode of \(G\) (\autocite[Theorem~4.10]{Wasilewski2024}).

In the classical case, a Cayley graph is built from a generating set which is closed under inverses and does not contain the unit. In the quantum case we have a projection \(p \in C(G)\) instead of a subset and these condition naturally translate to \(S(p) = p\) and \(\varepsilon(p) = 0\). Furthermore, we say that \(p\) is \emph{generating} if
\[
  \bigvee_{n \in \NN} [p^{*n}] = 1
\]
where \([x]\) denotes the range projection of an operator.

\begin{definition}
  Let \(G\) be a finite quantum group and \(p \in C(G)\) a generating projection satisfying \(S(p) = p\) and \(\varepsilon(p) = 0\). Then we define the \emph{quantum Cayley graph} of \(G\) with respect to \(p\) via the quantum adjacency operator
  \begin{align*}
    A \colon C(G) &\to C(G)\\
    x &\mapsto p * x.
  \end{align*}
\end{definition}

\begin{remark}
  Of course we can also consider projections \(p \in C(G)\) which do not fulfill the assumptions of the above definition. The corresponding convolution operators will still be quantum adjacency operators, but the resulting quantum graphs will generally not be undirected or connected and might have loops.
\end{remark}

\section{The Kac-Paljutkin Quantum Group}
The Kac-Paljutkin quantum group \(\KP\), introduced by Kac and Paljutkin in \cite{KacPaljutkin1966} (see \cite{Baraquin2019} and \autocite[Appendix~B]{FranzGohm2006} for a modern treatment), is the smallest quantum group which does not come from a group in the sense that its algebra of functions is neither commutative nor cocommutative. It is given by the \(8\)-dimensional \(*\)-algebra
\[
  C(\KP) = \CC \oplus \CC \oplus \CC \oplus \CC \oplus M_2(\CC)
\]
together with a comultiplication \(\Delta_\KP : C(\KP) \to C(\KP) \otimes C(\KP)\) defined as follows. If we denote the canonical basis of \(C(\KP)\) by
\begingroup
\allowdisplaybreaks
\begin{align*}
  e_1 =& 1 \oplus 0 \oplus 0 \oplus 0 \oplus
  \begin{pmatrix} 0 & 0 \\ 0 & 0
  \end{pmatrix}, &&a_{11} = 0 \oplus 0 \oplus 0 \oplus 0 \oplus
  \begin{pmatrix} 1 & 0 \\ 0 & 0
  \end{pmatrix}, \\
  e_2 =& 0 \oplus 1 \oplus 0 \oplus 0 \oplus
  \begin{pmatrix} 0 & 0 \\ 0 & 0
  \end{pmatrix}, &&a_{12} = 0 \oplus 0 \oplus 0 \oplus 0 \oplus
  \begin{pmatrix} 0 & 1 \\ 0 & 0
  \end{pmatrix}, \\
  e_3 =& 0 \oplus 0 \oplus 1 \oplus 0 \oplus
  \begin{pmatrix} 0 & 0 \\ 0 & 0
  \end{pmatrix}, &&a_{21} = 0 \oplus 0 \oplus 0 \oplus 0 \oplus
  \begin{pmatrix} 0 & 0 \\ 1 & 0
  \end{pmatrix}, \\
  e_4 =& 0 \oplus 0 \oplus 0 \oplus 1 \oplus
  \begin{pmatrix} 0 & 0 \\ 0 & 0
  \end{pmatrix}, &&a_{22} = 0 \oplus 0 \oplus 0 \oplus 0 \oplus
  \begin{pmatrix} 0 & 0 \\ 0 & 1
  \end{pmatrix},
\end{align*}
then \(\Delta_\KP\) is given by
\begin{align*}
  \Delta_\KP(e_1) &= e_1 \otimes e_1 + e_2 \otimes e_2 + e_3 \otimes e_3 + e_4 \otimes e_4 + \frac{1}{2} \sum_{1 \leq i,j \leq 2} a_{ij} \otimes a_{ij}, \\
  \Delta_\KP(e_2) &= e_1 \otimes e_2 + e_2 \otimes e_1 + e_3 \otimes e_4 + e_4 \otimes e_3 \\& + \frac{1}{2} (a_{11} \otimes a_{22} + a_{22} \otimes a_{11} + i a_{21} \otimes a_{12} - i a_{12} \otimes a_{21})\\
  \Delta_\KP(e_3) &= e_1 \otimes e_3 + e_3 \otimes e_1 + e_2 \otimes e_4 + e_4 \otimes e_2 \\& + \frac{1}{2} (a_{11} \otimes a_{22} + a_{22} \otimes a_{11} - i a_{21} \otimes a_{12} + i a_{12} \otimes a_{21})\\
  \Delta_\KP(e_4) &= e_1 \otimes e_4 + e_4 \otimes e_1 + e_2 \otimes e_3 + e_3 \otimes e_2 \\& + \frac{1}{2} (a_{11} \otimes a_{11} + a_{22} \otimes a_{22} - a_{12} \otimes a_{12} - a_{21} \otimes a_{21})\\
  \Delta_\KP(a_{11}) &= e_1 \otimes a_{11} + a_{11} \otimes e_1 + e_2 \otimes a_{22} + a_{22} \otimes e_2 \\& + e_3 \otimes a_{22} + a_{22} \otimes e_3 + e_4 \otimes a_{11} + a_{11} \otimes e_4 \\
  \Delta_\KP(a_{12}) &= e_1 \otimes a_{12} + a_{12} \otimes e_1 + i e_2 \otimes a_{21} - i a_{21} \otimes e_2 \\& - i e_3 \otimes a_{21} + i a_{21} \otimes e_3 - e_4 \otimes a_{12} - a_{12} \otimes e_4 \\
  \Delta_\KP(a_{21}) &= e_1 \otimes a_{21} + a_{21} \otimes e_1 - i e_2 \otimes a_{12} + i a_{12} \otimes e_2 \\& + i e_3 \otimes a_{12} - i a_{12} \otimes e_3 - e_4 \otimes a_{21} - a_{21} \otimes e_4 \\
  \Delta_\KP(a_{22}) &= e_1 \otimes a_{22} + a_{22} \otimes e_1 + e_2 \otimes a_{11} + a_{11} \otimes e_2 \\& + e_3 \otimes a_{11} + a_{11} \otimes e_3 + e_4 \otimes a_{22} + a_{22} \otimes e_4.
\end{align*}
\endgroup
The counit \(\eta_\KP : C(\KP) \to \CC\) is given by
\[
  \eta_\KP\Big(x_1 \oplus x_2 \oplus x_3 \oplus x_4 \oplus
    \begin{pmatrix} a_{11} & a_{12} \\ a_{21} & a_{22}
  \end{pmatrix}\Big) = x_1
\]
and the antipode \(S_\KP : C(\KP) \to C(\KP)\) is simply the transposition map, i.e. \(S_\KP(e_i) = e_i\) for \(i = 1, 2, 3, 4\) and \(S_\KP(a_{ij}) = a_{ji}\) for \(i,j = 1,2\). The Haar state \(h_\KP : C(\KP) \to \CC\) is given by
\[
  h_\KP\Big(x_1 \oplus x_2 \oplus x_3 \oplus x_4 \oplus
    \begin{pmatrix} a_{11} & a_{12} \\ a_{21} & a_{22}
  \end{pmatrix}\Big) = \frac{1}{8} (x_1 + x_2 + x_3 + x_4 + 2(a_{11} + a_{22})).
\]

Since our goal is to study quantum Cayley graphs associated to \(\KP\), we need to compute the convolutions of elements of \(C(\KP)\). In order to make the following computations less cumbersome and easier to follow, we give here a full table of all possible convolutions of basis elements of \(C(\KP)\).

\begin{table}[h]
  \centering
  \begin{tabular}{|c|c|c|c|c|c|c|c|c|}
    \hline
    \(x * y\) & \(e_1\) & \(e_2\) & \(e_3\) & \(e_4\) & \(a_{11}\) & \(a_{12}\) & \(a_{21}\) & \(a_{22}\) \\ \hline
    \(e_1\) & \(e_1\) & \(e_2\) & \(e_3\) & \(e_4\) & \(a_{11}\) & \(a_{12}\) & \(a_{21}\) & \(a_{22}\) \\ \hline
    \(e_2\) & \(e_2\) & \(e_1\) & \(e_4\) & \(e_3\) & \(a_{22}\) & \(i a_{21}\) & \(-i a_{12}\) & \(a_{11}\) \\ \hline
    \(e_3\) & \(e_3\) & \(e_4\) & \(e_1\) & \(e_2\) & \(a_{22}\) & \(-i a_{21}\) & \(i a_{12}\) & \(a_{11}\) \\ \hline
    \(e_4\) & \(e_4\) & \(e_3\) & \(e_2\) & \(e_1\) & \(a_{11}\) & \(-a_{12}\) & \(-a_{21}\) & \(a_{22}\) \\ \hline
    \(a_{11}\) & \(a_{11}\) & \(a_{22}\) & \(a_{22}\) & \(a_{11}\) & \(2(e_1 + e_4)\) & \(0\) & \(0\) & \(2(e_2 + e_3)\) \\ \hline
    \(a_{12}\) & \(a_{12}\) & \(-i a_{21}\) & \(i a_{21}\) & \(-a_{12}\) & \(0\) & \(2(e_1 - e_4)\) & \(2i(e_2 - e_3)\) & \(0\) \\ \hline
    \(a_{21}\) & \(a_{21}\) & \(i a_{12}\) & \(-i a_{12}\) & \(-a_{21}\) & \(0\) & \(-2i(e_2 - e_3)\) & \(2(e_1 - e_4)\) & \(0\) \\ \hline
    \(a_{22}\) & \(a_{22}\) & \(a_{11}\) & \(a_{11}\) & \(a_{22}\) & \(2(e_2 + e_3)\) & \(0\) & \(0\) & \(2(e_1 + e_4)\) \\ \hline
  \end{tabular}
\end{table}

We shall only be interested in irreflexive undirected and connected quantum graphs. According to \cite{Wasilewski2024}, a quantum Cayley graph has these properties if and only if the generating projection \(p\) satisfies \(\eta(p) = 0\) (irreflexive), \(R(p) = p\) (undirected) and \(\bigvee_{n \in \mathbb N} [p^{*n}] = 1\) (connected) where \(\eta\) is the counit, \(R\) is the unitary antipode and \([T]\) denotes the range projection of an operator \(T\). In the case of \(\KP\), we can give a classification of all such projections as follows (note that the unitary antipode equals the antipode in this case since \(\KP\) is of Kac type).

\begin{proposition}\label{prop:projection_classification}
  Let \(p \in C(\KP)\) be a projection. Then the associated quantum Cayley graph is undirected, irreflexive and connected if and only if \(p\) is of the form
  \begin{enumerate}
    \item \(p = x e_2 + y e_3 + z e_4 +
        \begin{pmatrix}
          1 & 0 \\ 0 & 1
      \end{pmatrix}\)
      where \(x, y, z \in \{0,1\}\),
    \item \(p = x e_2 + y e_3 + z e_4 + \frac{1}{1 + c^2}
        \begin{pmatrix}
          1 & c \\ c & c^2
      \end{pmatrix}\)
      where \(x, y, z \in \{0,1\}\) and \(c \in \mathbb{R}\setminus\{0\}\),
    \item \(p = x e_2 + y e_3 + z e_4 + r\)
      where \(x, y, z \in \{0,1\}\), \(r \in \left\{
          \begin{pmatrix}
            1 & 0 \\ 0 & 0
          \end{pmatrix},
          \begin{pmatrix}
            0 & 0 \\ 0 & 1
          \end{pmatrix}
      \right\}\) and \(x + y \geq 1\).
  \end{enumerate}
\end{proposition}

\begin{definition}
  We call a projection \emph{admissible} if it is of the form described in the previous proposition and say it is of \emph{type (1), (2)} or \emph{(3)} if it is of the corresponding form.
\end{definition}

\begin{proof}[Proof of Proposition~\ref{prop:projection_classification}]
  Being a projection in \(C(\KP)\), \(p\) is of the form
  \[
    p = w e_1 + x e_2 + y e_3 + z e_4 + r
  \]
  for \(w, x, y, z \in \{0,1\}\) and some projection \(r \in M_2(\CC)\). The irreflexivity condition \(\eta_{\KP}(p) = 0\) forces \(w = 0\) and the undirectedness condition \(S_{\KP}(p) = p\) implies that \(r\) must be real symmetric. For the connectedness condition, we distinguish the three cases \(\rank r = 0\), \(\rank r = 1\), and \(\rank r = 2\).

  For the case \(\rank r = 0\), note that \(\linearspan \{e_1, e_2, e_3, e_4\}\) is closed under convolution, so in this case \(p\) can never be generating.

  In the rank one case, \(r\) is either one of \(a_{11}\), \(a_{22}\), or of the form \(r = \frac{1}{1 + c^2}
    \begin{pmatrix} 1 & c \\ c & c^2
  \end{pmatrix}\) for some \(c \in \mathbb{R}\setminus\{0\}\). We assume now that \(r = a_{11}\). It is easy to see from the convolution table that \(\linearspan \{e_1, e_4, a_{11}\}\) is closed under convolution, so if \(x = y = 0\), \(p\) is not generating. Under the condition \(x + y \geq 1\) we have
  \begin{align*}
    p^{*2} =& \ x e_1 + x y e_4 + x z e_3 + x a_{22} \\
    &+ y x e_4 + y e_1 + y z e_2 + y a_{22} \\
    &+ z x e_3 + z y e_2 + z e_1 + z a_{11} \\
    &+ (x + y) a_{22} + z a_{11} \\
    &+ 2(e_1 + e_4) \\
    \geq& \ e_1 + e_4 + a_{22}
  \end{align*}
  and since left- (as well as right-)convolution with a projection is a (completely) positive map by \autocite[Proposition~4.9]{Wasilewski2024} we have
  \begin{align*}
    p^{*3} &\geq p * (e_1 + e_4 + a_{22}) \\
    &= p + x e_3 + y e_2 + z e_1 + a_{11} + (x + y) a_{11} + z a_{22} + 2(e_2 + e_3)\\
    &\geq e_2 + e_3.
  \end{align*}
  Hence \([p^{*1}] \vee [p^{*2}] \vee [p^{*3}] \geq a_{11} \vee (e_1 + e_4 + a_{22}) \vee (e_2 + e_3) = 1\).  Since the convolution table is symmetric under the swap \(a_{11} \leftrightarrow a_{22}\), the same holds in the case of \(r = a_{22}\).

  Suppose now that
  \[
    r = \frac{1}{1+c^2}
    \begin{pmatrix}
      1 & c \\ c & c^2
    \end{pmatrix}, \qquad c \neq 0,
  \]
  and put \(R \coloneqq (1+c^2) r =
    \begin{pmatrix}
      1 & c \\ c & c^2
  \end{pmatrix}\). The convolution table gives
  \begin{align*}
    R^{*2} =& \ 2(e_1 + e_4) + 2c^2(e_2 + e_3) + 2c^2(e_1 - e_4 + i(e_2 - e_3)) \\
    &+ 2c^2(e_1 - e_4 - i(e_2 - e_3)) + 2c^2(e_2 + e_3) + 2c^4(e_1 + e_4) \\
    =& \ 2\big((c^4 + 2c^2 + 1)e_1 + 2c^2 e_2 + 2c^2 e_3 + (c^4 - 2c^2 + 1)e_4\big) \\
    =& \ 2\big((c^2 + 1)^2 e_1 + 2c^2 e_2 + 2c^2 e_3 + (c^2 - 1)^2 e_4\big)
  \end{align*}
  and
  \begin{align*}
    R^{*3} =& \ 2\big( (c^2+1)^2 a_{11} + (c^2+1)^2 c a_{12} + (c^2+1)^2 c a_{21} + (c^2+1)^2 c^2 a_{22} \\
      &+ 2c^2 a_{22} + 2c^3 i a_{21} - 2c^3 i a_{12} + 2c^4 a_{11} \\
      &+ 2c^2 a_{22} - 2c^3 i a_{21} + 2c^3 i a_{12} + 2c^4 a_{11} \\
    &+ (c^2-1)^2 a_{11} - (c^2-1)^2 c a_{12} - (c^2-1)^2 c a_{21} + (c^2-1)^2 c^2 a_{22} \big) \\
    =& \ 2 \big( ((c^2+1)^2 + 4c^4 + (c^2-1)^2) a_{11} + ((c^2+1)^2 c - (c^2-1)^2 c) a_{12} \\
    &+ ((c^2+1)^2 c - (c^2-1)^2 c) a_{21} + ((c^2+1)^2 c^2 + 4c^2 + (c^2-1)^2 c^2) a_{22} \big) \\
    =& \ 4
    \begin{pmatrix}
      3c^4 + 1 & 2c^3 \\ 2c^3 & c^6 + 3c^2
    \end{pmatrix}.
  \end{align*}
  The determinant of the latter matrix is \(3c^2(c^4+1)^2 > 0\), so \([R^{*3}] = I_2\). If \(c \neq \pm 1\), then \([R^{*2}] = e_1 + e_2 + e_3 + e_4\), and hence \([R^{*2}] \vee [R^{*3}] = 1\). If \(c = \pm 1\), then
  \[
    R^{*3} = 4
    \begin{pmatrix}
      4 & \pm 2 \\ \pm 2 & 4
    \end{pmatrix}
  \]
  and a further convolution gives
  \begin{align*}
    R^{*4} =& \ 8\big( 4(e_1 + e_4) + 4(e_2 + e_3) + 2(e_1 - e_4) + 2i(e_2 - e_3) \\
    &+ 2(e_1 - e_4) - 2i(e_2 - e_3) + 4(e_1 + e_4) + 4(e_2 + e_3) \big) \\
    =& 96 e_1 + 64 e_2 + 64 e_3 + 32 e_4
  \end{align*}
  and \([R^{*3}] \vee [R^{*4}] = 1\). Since multiplication with a non-zero scalar does not change the support projection, we thus have that \(r\) is generating for all \(c \neq 0\). Write \(p = s + r\) where \(s = x e_2 + y e_3 + z e_4\) and denote the corresponding left-convolution operators by \(T_p, T_s, T_r\) respectively. Expanding \(T_p^n - T_r^n = (T_s + T_r)^n - T_r^n\) as a sum of compositions of completely positive maps shows that \(T_p^n - T_r^n\) is completely positive and hence
  \[
    p^{*n} = T_p^n(e_1) \geq T_r^n(e_1) = r^{*n},
  \]
  so \(p\) is generating as well.

  Finally, if \(\rank r = 2\), then \(r =
    \begin{pmatrix} 1 & 0 \\ 0 & 1
  \end{pmatrix}\) and
  \begin{align*}
    p^{*2} =& \ x e_1 + x y e_4 + x z e_3 + x
    \begin{pmatrix} 1 & 0 \\ 0 & 1
    \end{pmatrix}\\
    &+ y x e_4 + y e_1 + y z e_2 + y
    \begin{pmatrix} 1 & 0 \\ 0 & 1
    \end{pmatrix}\\
    &+ z x e_3 + z y e_2 + z e_1 + z
    \begin{pmatrix} 1 & 0 \\ 0 & 1
    \end{pmatrix}\\
    &+ (x + y + z)
    \begin{pmatrix} 1 & 0 \\ 0 & 1
    \end{pmatrix}\\
    &+ 4(e_1 + e_2 + e_3 + e_4)
  \end{align*}
  so \([p^{*1}] \vee [p^{*2}] = 1\).
\end{proof}

\begin{proposition}\label{prop:spectrum}
  Let
  \[
    p = x e_2 + y e_3 + z e_4 +
    \begin{pmatrix}
      a & b \\ b & d
    \end{pmatrix}
  \]
  be an admissible projection, where \(x,y,z \in \{0,1\}\) and \(a,b,d \in \RR\). Put
  \begin{align*}
    &\lambda_0 = x + y + z + 2(a+d), &&\lambda_1 = -x - y + z + 2(a-d)\\
    &\lambda_2 = -x - y + z - 2(a-d), &&\lambda_3 = x + y + z - 2(a+d),
  \end{align*}
  and
  \[
    M_p =
    \begin{pmatrix}
      x-y-z & 2 \sqrt{2}b\\
      2\sqrt{2}b & -x+y-z
    \end{pmatrix}
  \]
  and let \(\mu_+, \mu_- \in \RR\) be the two eigenvalues of \(M_p\). Then the spectrum of the quantum adjacency operator \(A_p\) associated to \(p\) is given by
  \[
    \{\lambda_0, \lambda_1, \lambda_2, \lambda_3, \mu_+, \mu_-\}
  \]
  where the \(\lambda_i\)'s each have multiplicity one and \(\mu_-, \mu_+\) each have multiplicity two. More explicitly:
  \begin{enumerate}[(1)]
    \item For \(p\) of type (1), the spectrum of \(A_p\) is
      \[
        \{x+y+z+4, x+y+z-4, -x-y+z, x-y-z, -x+y-z\}.
      \]
    \item For \(p\) of type (2), the spectrum of \(A_p\) is
      \[
        \{x+y+z+2, x+y+z-2, -x-y+z+2\rho_c, -x-y+z-2\rho_c, \alpha_{c, +}, \alpha_{c, -}\}
      \]
      where
      \[
        \rho_c = \frac{1-c^2}{1+c^2}, \qquad \alpha_{c, \pm} = -z \pm \sqrt{(x-y)^2 + \frac{8c^2}{(1+c^2)^2}}
      \]
      and \(c \in \RR\) such that
      \[
        \begin{pmatrix}
          a & b \\ b & d
        \end{pmatrix}
        = \frac{1}{1+c^2}
        \begin{pmatrix}
          1 & c \\ c & c^2
        \end{pmatrix}.
      \]
    \item For \(p\) of type (3), the spectrum of \(A_p\) is
      \[
        \{x+y+z+2, x+y+z-2, -x-y+z+2, -x-y+z-2, x-y-z, -x+y-z\}.
      \]
  \end{enumerate}
\end{proposition}
\begin{proof}
  We decompose the left regular representation of the convolution algebra \((C(\KP), *)\) into a direct sum of irreducible representations. For \(\alpha, \beta \in \{1, -1\}\), define one-dimensional representations \(\chi_{\alpha, \beta}\) by
  \begin{align*}
    &\chi_{\alpha, \beta}(e_1) = \chi_{\alpha, \beta}(e_4) = 1, && \chi_{\alpha, \beta}(e_2) = \chi_{\alpha, \beta}(e_3) = \alpha, &&&\\
    &\chi_{\alpha, \beta}(a_{11}) = 2\beta, && \chi_{\alpha, \beta}(a_{22}) = 2\alpha\beta, &&&\chi_{\alpha, \beta}(a_{12}) = \chi_{\alpha, \beta}(a_{21}) = 0.
  \end{align*}
  The remaining irreducible representation is the two-dimensional representation \(\pi\) given by
  \begin{align*}
    &\pi(e_1) = I_2, \qquad \pi(e_2) = \sigma_z, \qquad \pi(e_3) = -\sigma_z, \qquad \pi(e_4) = -I_2,\\
    &\pi(a_{11}) = \pi(a_{22}) = 0, \qquad \pi(a_{12}) = \sqrt{2}(\sigma_x - \sigma_y), \qquad \pi(a_{21}) = \sqrt{2}(\sigma_x + \sigma_y)
  \end{align*}
  where \(\sigma_x, \sigma_y, \sigma_z \in M_2(\CC)\) are the Pauli matrices. Although tedious --- it is easy to verify, using the convolution table above, that these maps are indeed irreducible representations. For dimension reasons, they form a complete set of irreducible representations and therefore the left regular representation decomposes as
  \[
    \Big(\bigoplus_{\alpha,\beta \in \{1,-1\}} \chi_{\alpha, \beta} \Big) \oplus \pi \oplus \pi.
  \]
  For the spectrum of \(A_p\) this means that
  \[
    \sigma(A_p) = \Big( \bigcup_{\alpha, \beta \in \{1,-1\}} \sigma(\chi_{\alpha, \beta}(p)) \Big) \cup \sigma(\pi(p)) \cup \sigma(\pi(p)).
  \]
  Evaluating the four characters \(\chi_{\alpha, \beta}\) at \(p\) gives \(\lambda_0, \lambda_1, \lambda_2, \lambda_3\) whereas \(\pi(p) = M_p\). This proves the claim. Substituting the different forms of \(p\) for the three differents types yields the more explicit formulas stated above.
\end{proof}

\begin{proposition}\label{prop:aut_isomorphisms}
  Let
  \[
    J =
    \begin{pmatrix}
      0 & 1 \\ 1 & 0
    \end{pmatrix}, \qquad
    D =
    \begin{pmatrix}
      1 & 0 \\ 0 & -1
    \end{pmatrix}
  \]
  and define \(*\)-automorphisms \(\beta, \gamma\) of \(C(\KP)\) by
  \begin{align*}
    &\beta(e_1) = e_1, \quad \beta(e_2) = e_3, \quad \beta(e_3) = e_2, \quad \beta(e_4) = e_4,  &&\beta(x) = J x J\\
    &\gamma(e_k) = e_k \quad (1 \leq k \leq 4), &&\gamma(x) = D x D
  \end{align*}
  for \(x \in M_2(\CC)\). Then \(\beta\) and \(\gamma\) are commuting Hopf \(*\)-automorphisms and for every admissible projection \(p \in C(\KP)\), the quantum Cayley graphs associated with \(p\), \(\beta(p)\), \(\gamma(p)\) and \(\beta \gamma (p)\) are isomorphic.

  On the parameters of admissible projections, this gives the following identifications:
  \begin{enumerate}[(1)]
    \item in type (1), \((x,y,z) \sim (y,x,z)\),
    \item in type (2),
      \[
        (x,y,z,c) \sim (x,y,z,-c) \sim (y,x,z,c^{-1}) \sim (y,x,z,-c^{-1}),
      \]
    \item in type (3), \((x,y,z,a_{11}) \sim (y,x,z,a_{22})\).
  \end{enumerate}
\end{proposition}
\begin{proof}
  Both \(\beta\) and \(\gamma\) are clearly \(*\)-automorphisms and they commute since \(J\) and \(D\) anticommute, so for any \(x \in M_2(\CC)\) we have
  \[
    \beta \gamma(x) = JDxDJ = (-DJ) x (-JD) = DJxJD = \gamma \beta (x).
  \]
  Moreover, it is easy to see that \(\Delta_\KP \beta = (\beta \otimes \beta) \Delta_\KP\) and \(\Delta_\KP \gamma = (\gamma \otimes \gamma) \Delta_\KP\) by checking equality on the basis vectors. Both \(\beta\) and \(\gamma\) also clearly preserve the counit and the Haar state, so they are Hopf \(*\)-automorphisms. Hence they preserve the convolution product and therefore we have for any \(\theta \in \{\beta, \gamma, \beta\gamma\}\)
  \[
    A_{\theta(p)}(\theta(x)) = \theta(p) * \theta(x) = \theta(p * x) = \theta(A_p(x))
  \]
  so \(\theta\) is a quantum graph isomorphism. The parameter identifications follow from \(J r_c J = r_{1/c}\) and \(D r_c D = r_{-c}\) where
  \[
    r_c = \frac{1}{1+c^2}
    \begin{pmatrix}
      1 & c \\ c & c^2
    \end{pmatrix}.
  \]
\end{proof}

It turns out that the in the type (2) case, all possible quantum graph isomorphisms are realized by the Hopf \(*\)-automorphisms above, but in the type (1) and type (3) cases, there are more identifications possible that do not come from a Hopf \(*\)-automorphism. The following theorem gives the full classification of isomorphisms of quantum Cayley graphs of \(\KP\).

\begin{theorem}\label{thm:isomorphism_classification}
  Let
  \[
    p = x e_2 + y e_3 + z e_4 + r, \qquad q = x' e_2 + y' e_3 + z' e_4 + r'
  \]
  be two admissible projections. The associated quantum Cayley graphs are isomorphic if and only if \(p\) and \(q\) are of the same type and the corresponding condition below holds.
  \begin{enumerate}[(1)]
    \item If \(p\) and \(q\) are of type (1), then
      \[
        x + y + z = x' + y' + z'.
      \]
    \item If \(p\) and \(q\) are of type (2) and
      \[
        r = \frac{1}{1+c^2}
        \begin{pmatrix}
          1 & c \\ c & c^2
        \end{pmatrix}, \qquad r' = \frac{1}{1+c'^2}
        \begin{pmatrix}
          1 & c' \\ c' & c'^2
        \end{pmatrix},
      \]
      then either
      \[
        (x, y, z) = (x', y', z') \quad \textrm{and} \quad c' \in \{c, -c\}
      \]
      or
      \[
        (x, y, z) = (y', x', z') \quad \textrm{and} \quad c' \in \{c^{-1}, -c^{-1}\}.
      \]
    \item If \(p\) and \(q\) are of type (3), then
      \[
        z = z' \quad \textrm{and} \quad x + y = x' + y'.
      \]
  \end{enumerate}
\end{theorem}
\begin{proof}
  We identify \(\{1,2,3,4\}\) with the Klein four group \(V \cong \ZZ_2^2\) in such a way that \(1\) corresponds to the identity element and the convolution of basis vectors in \(\CC^4\) corresponds to multiplication in the group, \ie \(e_j * e_k = e_{jk}\) (where the multiplication \(jk\) is in \(V\)). We denote the sets of non-zero indices in the \(\CC^4\)-part of \(p\) and \(q\) by \(S \coloneqq \{j \in \{2,3,4\} \mid e_j p = e_j\}\) and \(S' \coloneqq \{j \in \{2,3,4\} \mid e_j q = e_j\}\). Every \(*\)-automorphism of \(C(\KP)\) is of the form \(\theta_{\sigma, U} = P_\sigma \oplus \Ad_U\) for some permutation \(\sigma \in \{1,2,3,4\}\) and a unitary \(U \in M_2(\CC)\). We write \(\sigma(k) = g \overline{\sigma}(k)\) for unique \(g \in V\) and \(\overline{\sigma} \in \operatorname{Aut}(V) \cong S_3\). The Pauli matrices \(\sigma_x, \sigma_y, \sigma_z\) form a basis of the self-adjoint traceless matrices in \(M_2(\CC)\) so we can identify this subspace with \(\RR^3\). If we set
  \begin{align*}
    u_2 &= \frac{1}{\sqrt{2}} (\sigma_x + \sigma_y) \cong \frac{1}{\sqrt{2}}(1,1,0) \in \RR^3\\
    u_3 &= \frac{1}{\sqrt{2}} (\sigma_x - \sigma_y) \cong \frac{1}{\sqrt{2}}(1,-1,0) \in \RR^3\\
    u_4 &= \sigma_z \cong (0,0,1) \in \RR^3
  \end{align*}
  then the convolution table gives
  \[
    e_j * I_2 = I_2, \qquad Q_j \coloneqq (e_j * \cdot)\vert_{\RR^3} = 2 u_j u_j^\top - I_3, \qquad T_S \coloneqq \sum_{j \in S} Q_j = 2 P_{E_S} - \lvert S \rvert I_3,
  \]
  where \(P_{E_S}\) is the orthogonal projection onto \(E_S \coloneqq \linearspan_\RR \{u_j \mid j \in S\}\). Let \(R \in \operatorname{SO}(3)\) be the rotation induced by \(\Ad_U\) and put \(r_k \coloneqq r * e_k\) and \(r_k' \coloneqq r' * e_k\). Comparing the blocks of the intertwining relation
  \begin{equation}
    \theta_{\sigma, U} A_p = A_q \theta_{\sigma, U} \label{eq:iso_intertwining}
  \end{equation}
  relative to the decomposition \(C(\KP) = \CC^4 \oplus M_2(\CC)\) yields
  \begin{align}
    S' &= \overline{\sigma}(S), \label{eq:iso_rel1}\\
    R T_S &= T_{S'} R, \label{eq:iso_rel2}\\
    U r_k U^* &= r_{\sigma(k)}', \qquad (1 \leq k \leq 4). \label{eq:iso_rel3}
  \end{align}
  Indeed, in the scalar block of \(A_p\) we have the adjacency matrix \(C_S\) of the Cayley graph of \(V\) with respect to \(S\) and \(P_\sigma C_S = C_{S'} P_\sigma\) is equivalent to \(S' = \overline{\sigma}(S)\), which implies \(\lvert S \rvert = \lvert S' \rvert\), so the equality (\ref{eq:iso_intertwining}) on the subspace \(\CC I_2\) follows from this already. By complex linearity, we can reduce the \(M_2(\CC)\)-block to \(\RR^3\), \ie (\ref{eq:iso_rel2}). The off-diagonal blocks only give one equation since they are adjoints of each other by the self-adjointness of \(A_p\) and \(A_q\).

  \vspace*{1em}

  The third equation (\ref{eq:iso_rel3}) excludes isomorphisms between different types since in type (1), all \(r_k\) are equal to \(I_2\), in type (2) the \(r_k\) are four distinct rank-one projections and in type (3) the \(r_k\) are two sets of two rank-one projections.

  \vspace*{1em}

  Now let \(p\) and \(q\) both be of type (1). Then (\ref{eq:iso_rel3}) is automatic and an isomorphism exists if and only if \(\lvert S \rvert = \lvert S' \rvert\), because in that case we can choose a \(\sigma \in S_4\) with \(\overline{\sigma}(S) = S'\) and \(R \in \operatorname{SO}(3)\) with \(R E_S = E_{S'}\) (which is equivalent to (\ref{eq:iso_rel2})) and lift it to a unitary \(U \in M_2(\CC)\).

  \vspace*{1em}

  Let \(p\) and \(q\) both be of type (2) and assume they yield isomorphic quantum graphs. The Bloch vectors in \(\RR^3\) corresponding to \(r_1, r_2, r_3, r_4\) are the columns of
  \[
    \Phi_c =
    \begin{pmatrix}
      \xi_c & 0 & 0 & -\xi_c\\
      0 & -\xi_c & \xi_c & 0\\
      \zeta_c & -\zeta_c & -\zeta_c & \zeta_c
    \end{pmatrix}
  \]
  where
  \[
    \xi_c = \frac{2c}{1+c^2}, \qquad \zeta_c = \frac{1-c^2}{1+c^2}.
  \]
  The unitary conjugation in (\ref{eq:iso_rel3}) corresponds to a rotation on the Bloch sphere, so we can rewrite it as
  \[
    R\Phi_c = \Phi_{c'} P_\sigma.
  \]
  Because both \(R\) and \(P_\sigma\) are orthogonal, this means that \(\Phi_c\) and \(\Phi_{c'}\) must have the same multiset of singular values, which is given by
  \[
    \{\{2\xi_c^2, 2\xi_c^2, 4\zeta_c^2\}\}.
  \]
  The multiplicity determines \(\xi_c^2\) unless all three values coincide, in which case \(\xi_c^2 = 2/3\). Thus, equality of multisets for \(c\) and \(c'\) implies \(\xi_c^2 = \xi_{c'}^2\) and \(\zeta_c^2 = \zeta_{c'}^2\). Since \(t \mapsto (1-t)/(1+t)\) is injective on \((0,\infty)\) and \(\zeta_{1/c} = -\zeta_c\), we infer that \(c' \in \{c, -c, c^{-1}, -c^{-1}\}\). Using Proposition~\ref{prop:aut_isomorphisms} we may now replace \(q\) by some isomorphic projection and assume \(c = c'\), replacing \(S'\) with \(\tau(S')\) (where \(\tau = (2 \ 3)\)) if \(c' \in \{c^{-1}, -c^{-1}\}\).

  Suppose for now that \(\zeta_c \neq 0\). On \(K \coloneqq (1,1,1,1)^\perp\) the map \(\Phi_c\) is invertible, so taking determinant in \(R \Phi_c\vert_K = \Phi_c\vert_K P_\sigma\vert_K\) gives \(\sgn(\sigma) = 1\). Since left-translations of the Klein four group induce even permutations, we have \(\sgn(\overline{\sigma}) = 1\) as a permutation of \(\{2,3,4\}\). We have
  \[
    Q_j \Phi_c = \Phi_c C_{\{\tau(j)\}}
  \]
  and by summing over \(j \in S\)
  \[
    T_S \Phi_c = \Phi_c C_{\tau(S)}.
  \]
  Combining this with \(R \Phi_c = \Phi_c P_\sigma\) and \(RT_S = T_{S'}R\) gives
  \[
    \Phi_c P_\sigma C_{\tau(S)} = \Phi_c C_{\tau(S')}P_\sigma.
  \]
  On \(K\) we can cancel \(\Phi_c\) and get \(P_\sigma C_{\tau(S)} = C_{\tau(S')} P_\sigma\), but this also holds on the subspace spanned by \((1,1,1,1)\) because \(\lvert S \rvert = \lvert S' \rvert\). But this is equivalent to \(\overline{\sigma}\tau(S) = \tau(S')\) so together with (\ref{eq:iso_rel1}) we get
  \[
    \overline{\sigma}(S) = \tau \overline{\sigma} \tau (S) = \overline{\sigma}^{-1}(S).
  \]
  (Note that \(\tau \overline{\sigma} \tau = \overline{\sigma}^{-1}\) since \(\overline{\sigma}\) is even.) Since \(\overline{\sigma}^3 = \id\) this implies that \(\overline{\sigma}(S) = S\) and hence \(S = S'\).

  If \(\zeta_c = 0\), then \(c = \pm 1\) and \(\ker(\Phi_c\vert_K) = \linearspan_\RR (1,-1,-1,1)\). Hence \(P_\sigma\) must preserve this subspace, so \(\sigma\) preserves the unordered partition \(\{\{1,4\}, \{2,3\}\}\). But then \(\overline{\sigma} (4) = 4\), so \(\overline{\sigma} \in \{\id, \tau\}\). Therefore \(S' = S\) or \(S' = \tau(S)\).

  Together with the normalization \(c = c'\) this yields exactly the two alternatives in the statement of the theorem. The reverse direction follows from Proposition~\ref{prop:aut_isomorphisms}.

  \vspace*{1em}

  Let \(p\) and \(q\) both be of type (3). Then we have \(r_1 = r_4\) and \(r_2 = r_3\) irrespective of the choice of \(r\). So if there exists an isomorphism, we must have \(\overline{\sigma} \in \{\id, \tau\}\) (by similar arguments as in the type (2) case), yielding \(z = z'\) and \(x + y = x' + y'\). For the reverse direction observe that the unitary
  \[
    U_0 =
    \begin{pmatrix}
      0 & 1\\i & 0
    \end{pmatrix}
  \]
  exchanges \(a_{11}\) and \(a_{22}\) and the corresponding rotation in \(\operatorname{SO}(3)\) is
  \[
    Q_2 =
    \begin{pmatrix}
      0 & 1 & 0\\
      1 & 0 & 0\\
      0 & 0 & -1
    \end{pmatrix}
  \]
  which commutes with both \(Q_3\) and \(Q_4\), so with every \(T_S\). Hence \(\theta_{\id, U_0}\) is an isomorphism which exchanges the two possible choices of \(r\) while leaving \(S\) fixed. Concatenation with \(\beta\) from Proposition~\ref{prop:aut_isomorphisms} realizes all parameter identifications in the type (3) case.
\end{proof}

\begin{theorem}\label{thm:chromatic_numbers}
  Let
  \[
    p = x e_2 + y e_3 + z e_4 + r \in C(\KP)
  \]
  be an admissible projection, denote the associated quantum Cayley graph by \(G_p\) and put
  \[
    S \coloneqq \{j \in \{2,3,4\} \mid \textrm{the } e_j \textrm{ component of } p \textrm{ is } 1\}.
  \]
  Then the chromatic number of \(G_p\) is
  \[
    \chi(G_p) =
    \begin{cases}
      \infty, & \textrm{if } S = \{2,3,4\},\\
      2, & \textrm{if } S = \emptyset,\\
      2, & \textrm{if } p \textrm{ is of type (3) and } z = 0,\\
      4, & \textrm{otherwise.}
    \end{cases}
  \]
\end{theorem}
\begin{proof}
  Put \(s \coloneqq \sum_{j \in S} e_j\) such that \(p = s + r\) and for a projection \(P \in C(\KP)\) write \(P = t + q\) with \(t = \sum_{j \in T} e_j\) for some \(T \subseteq \{1,2,3,4\}\) and a projection \(q \in M_2(\CC)\). Denote the quantum adjacency operator corresponding to \(p\) by \(A_p\). We use the characterization of colorings from Lemma~\ref{lem:coloring_equivalence}. Since \(A_p(P) \geq 0\), we have \(P A_p(P) = 0\) if and only if \(P A_p(P) P = 0\). We have
  \begin{align}
    P A_p(P) P &= (t + q) \big[ (s + r) * (t + q) \big] (t + q) \notag\\
    &= (t + q) \big[ \underbrace{s * t + r * q}_{\in \CC^4} + \underbrace{s * q + r * t}_{\in M_2(\CC)} \big] (t + q) \notag\\
    &= \sum_{l \in T} \big(m(l) + \langle e_l, r * q \rangle \big) e_l + \sum_{j \in S} q(e_j * q)q + \sum_{k \in T} q(r * e_k)q \notag\\
    &= \sum_{l \in T} \big(m(l) + 2 \Tr((r * e_l)q) \big) e_l + \sum_{j \in S} q(e_j * q)q + \sum_{k \in T} q(r * e_k)q \label{eq:coloring_sum}
  \end{align}
  where \(m(l) = \lvert \{(i,j) \in S \times T \mid ij = l\} \rvert\) and the multiplication \(ij\) is in the Klein four group as in the proof of Theorem~\ref{thm:isomorphism_classification}. Since all of the summands above are non-negative, it holds that \(P A_p(P) P = 0\) if and only if every term in (\ref{eq:coloring_sum}) vanishes. Therefore, we have that \(P A_p(P) = 0\) if and only if the following holds:
  \begin{enumerate}[(i)]
    \item \(m(l) = 0\) for all \(l \in T\), or equivalently, \(T\) is an independent set in the classical Cayley graph of \(V\) with generating set \(S\);
    \item \(q(e_j * q)q = 0\) for all \(j \in S\);
    \item \(q(r * e_k)q = 0\) for all \(k \in T\).
  \end{enumerate}
  Note that condition (iii) already implies that \(\Tr((r*e_k)q) = 0\).

  We make condition (ii) more explicit: It is clear that \(q ( e_j * q) q \neq 0\) if \(q = I_2\), so it suffices to consider rank one projections for \(q\). Then \(e_j * q\) is also a rank one projection and if we write \(q = \frac{1}{2}(I_2 + c_x \sigma_x + c_y \sigma_y + c_z \sigma_z)\) in the Pauli basis then
  \begin{align*}
    e_2 * q &= \frac{1}{2}(I_2 + c_y \sigma_x + c_x \sigma_y - c_z \sigma_z),\\
    e_3 * q &= \frac{1}{2}(I_2 - c_y \sigma_x - c_x \sigma_y - c_z \sigma_z),\\
    e_4 * q &= \frac{1}{2}(I_2 - c_x \sigma_x - c_y \sigma_y + c_z \sigma_z).
  \end{align*}
  Since two rank one projections in \(M_2(\CC)\) are orthogonal if and only if their Bloch vectors are antipodal, we see that
  \begin{align}
    q(e_2 * q)q = 0 &\iff c_x = -c_y \notag\\
    q(e_3 * q)q = 0 &\iff c_x = c_y\label{eq:coloring_condition}\\
    q(e_4 * q)q = 0 &\iff c_z = 0.\notag
  \end{align}
  Thus if \(S = \{2,3,4\}\), then \(q = 0\), so no coloring can exist in that case, \ie \(\chi(G_p) = \infty\).

  If \(S \subseteq \{2,3,4\}\) is a non-empty proper subset, the following table gives Pauli coefficients satisfying (\ref{eq:coloring_condition}) and an independent set \(X\) in the Cayley graph \(\Cay(V, S)\):
  \begin{center}
    \begin{tabular}{ c | c | c }
      \(S\) & \((c_x, c_y, c_z)\) & \(X\) \\
      \hline
      \(\{2\}\) & \((0,0,1)\) & \(\{1,4\}\) \\
      \(\{3\}\) & \((0,0,1)\) & \(\{1,4\}\) \\
      \(\{4\}\) & \((1,0,0)\) & \(\{1,2\}\) \\
      \(\{2,3\}\) & \((0,0,1)\) & \(\{1,4\}\) \\
      \(\{2,4\}\) & \(\frac{1}{\sqrt{2}}(1,-1,0)\) & \(\{1,3\}\) \\
      \(\{3,4\}\) & \(\frac{1}{\sqrt{2}}(1,1,0)\) & \(\{1,2\}\)
    \end{tabular}
  \end{center}
  So in all of these cases, a four coloring is given by
  \begin{equation}
    q, \qquad I_2 - q, \qquad \sum_{j \in X} e_j, \qquad \sum_{j \notin X} e_j \label{eq:coloring}
  \end{equation}
  where \(q \in M_2(\CC)\) is the rank projection corresponding to the above coefficients in the Pauli basis (note that condition (iii) is always fulfilled here), hence \(\chi(G_p) \leq 4\). It remains to be seen in which cases a smaller coloring exists.

  If \(p\) is of type (1), then \(r = I_2\) and \(r * e_j = I_2\) for all \(j\). Condition (iii) says that a coloring projection can not be non-zero in the \(M_2\)-part and the scalar part at the same time. The Cayley graph \(\Cay(V, S)\) contains at least one edge, so we need at least two colors for the scalar part by condition (i). By condition (ii), we also need at least two colors for the \(M_2\)-part. Therefore, \(\chi(G_p) = 4\).

  Now let \(p\) be of type (2). As in the proof of Theorem~\ref{thm:isomorphism_classification}, we have that the Bloch vectors of \(r * e_j\) are given by
  \[
    (\xi_c, 0, \zeta_c), \quad (0, -\xi_c, -\zeta_c), \quad (0, \xi_c, -\zeta_c), \quad (-\xi_c, 0, \zeta_c)
  \]
  respectively, where
  \[
    \xi_c = \frac{2c}{1+c^2}, \qquad \zeta_c = \frac{1-c^2}{1+c^2}.
  \]
  Condition (iii) says that a coloring projection can be non-zero in the scalar part as well as in the \(M_2\) part only if the Bloch vector of the \(M_2\) part is antipodal to the vectors in the above list corresponding to the scalar parts. Comparing with (\ref{eq:coloring_condition}) shows that this can never happen except when \(S = \{4\}\) and \(c = \pm 1\). Even in this case, the only "mixed" coloring projections that are possible are
  \[
    e_1 + \frac{1}{2}(I_2 \pm \sigma_x), \quad e_2 + \frac{1}{2}(I_2 \mp \sigma_y), \quad e_3 + \frac{1}{2}(I_2 \pm \sigma_y), \quad e_4 + \frac{1}{2}(I_2 \mp \sigma_x),
  \]
  the sign depending on the sign of \(c\). Since \(\{1,4\}\) and \(\{2,3\}\) are edges (and not independent sets) in \(\Cay(V, \{4\})\) we still need four projections to color \(G_p\), so \(\chi(G_p) = 4\).

  Suppose that \(p\) is of type (3) and assume without loss of generality that \(r = a_{11}\) (the case \(r = a_{22}\) works analogously). Then \(r * e_1 = r * e_4 = a_{11}\) and \(r * e_2 = r * e_3 = a_{22}\). If \(z = 0\) (\ie \(4 \notin S\)), the projections
  \[
    e_2 + e_3 + r, \qquad e_1 + e_4 + (I_2 - r)
  \]
  form a two-coloring. If \(z = 1\) and \(x + y = 1\), then condition (\ref{eq:coloring_condition}) forces the matrix part of any coloring projection to have vanishing \(\sigma_z\)-coordinate and non-zero \(\sigma_x\)- and \(\sigma_y\)-coordinates, so it can not be \(a_{11}\) or \(a_{22}\). Thus, no coloring projection can be non-zero in the scalar as well as the matrix part. Hence, \(\chi(G_p) = 4\).

  Finally, if \(S = \emptyset\), then the projections
  \[
    e_1 + e_2 + e_3 + e_4, \qquad I_2
  \]
  form a two-coloring, so \(\chi(G_p) = 2\).
\end{proof}

\section{The Quantum Cuntz-Krieger algebra of \(\Cay(\KP, I_2)\)}
In this section we compute the quantum Cuntz-Krieger algebra associated to the quantum Cayley graph of \(\KP\) with generating projection \(p = I_2\).

(Free) Quantum Cuntz-Krieger algebras were introduced in \cite{BrannanEiflerVoigtWeber2022} in order to generalize the classical construction of Cuntz-Krieger algebras to quantum graphs. They were further studied in \cite{BrannanHamidiIsmertNelsonWasilewski2023} where also a unital variant was introduced. We give the definition of the unital version below.

\begin{definition}[\cite{BrannanHamidiIsmertNelsonWasilewski2023}]
  Let \(B = (B, \psi, A)\) be a tracial quantum graph. A \emph{quantum Cuntz-Krieger \(G\)-family} in a unital \(C^*\)-algebra \(D\) is given by a linear map \(s \colon B \to D\) such that
  \begin{enumerate}[(i)]
    \item \(\mu_D(\mu_D \otimes 1)(s \otimes s^* \otimes s)(m^* \otimes 1)m^* = s\),
    \item \(\mu_D(s^* \otimes s) m^* = \mu_D(s \otimes s^*)m^* A\),
    \item \(\mu_D(s \otimes s^*)m^*(1_B) = 1_D\),
  \end{enumerate}
  where \(\mu_D \colon D \otimes D \to D\) is the multiplication map of \(D\) and \(s^*(x) = s(x^*)^*\). The \emph{quantum Cuntz-Krieger algebra} associated to \(G\) is the universal unital \(C^*\)-algebra \(\mathcal{O}(G)\) generated by the image of a quantum Cuntz-Krieger \(G\)-family \(s \colon B \to \mathcal{O}(G)\).
\end{definition}

\begin{remark}
  If \(B = \bigoplus_{\alpha = 1}^d M_{n_\alpha}(\CC)\), then \(\mathcal{O}(G)\) can equivalently be described as the universal unital \(C^*\)-algebra generated by elements \(s_{ij}^{(\alpha)}\) (\(i, j = 1, \ldots n_\alpha\), \(\alpha = 1, \ldots, d\)) subject to the relations
  \begin{enumerate}[(i)]
    \item \(\displaystyle \sum_{k,l=1}^{n_\alpha} s_{ik}^{(\alpha)} (s_{lk}^{(\alpha)})^* s_{lj}^{(\alpha)} = s_{ij}^{(\alpha)}\),
    \item \(\displaystyle \sum_{l=1}^{n_\alpha} (s_{li}^{(\alpha)})^* s_{lj}^{(\alpha)} = \sum_{\beta = 1}^d \sum_{l,m,n = 1}^{n_\beta} A_{ij\alpha}^{lm\beta} s_{ln}^{(\beta)} (s_{mn}^{(\beta)})^*\),
    \item \(\displaystyle \sum_{\beta=1}^d \sum_{k,l=1}^{n_\beta} \psi(e_{jj}^{(\beta)}) s_{ij}^{(\beta)} (s_{ij}^{(\beta)})^* = 1\)
  \end{enumerate}
  for all \(i,j = 1, \ldots, n_\alpha\), \(\alpha = 1, \ldots, d\), where \(e_{ij}^{(\alpha)}\) denote the usual matrix units and \(A_{ij\alpha}^{kl\beta}\) are the coefficients of \(A\) with respect to the basis of adapted matrix units from \autocite[Lemma~3.2]{BrannanEiflerVoigtWeber2022}. Keep in mind that in our convention, the functional \(\psi\) (and consequently the induced inner product) has a different scaling than in \cite{BrannanEiflerVoigtWeber2022} and \cite{BrannanHamidiIsmertNelsonWasilewski2023}, so our relations are slightly different.
\end{remark}

Now let \(A \colon C(\KP) \to C(\KP)\) be the quantum adjacency matrix associated to the generating projection \(p = I_2\). The coefficients of \(A\) in the basis of adapted matrix units \(f_i = e_i\), \(f_{ij} = a_{ij}/2\) are
\[
  A(f_i) = 2(f_{11} + f_{22}), \qquad A(f_{ij}) = \delta_{ij} \sum_{k=1}^{4} f_k.
\]
Hence the associated quantum Cuntz-Krieger algebra \(\mathcal{O}(G)\) is generated by elements
\begin{align*}
  &s_i, \quad &&(i=1, \ldots, 4)\\
  &t_{ij}, \quad &&(i,j =1, 2)
\end{align*}
subject to the relations
\begin{align*}
  s_i s_i^* s_i &= s_i, && \forall i=1, \ldots, 4\\
  tt^*t &= t,\\
  s_i^* s_i &= 2 \sum_{k,l=1}^2 t_{kl} t_{kl}^*, &&\forall i=1,\ldots,4\\
  [t^*t]_{ij} &= \delta_{ij} \sum_{k=1}^4 s_k s_k^*, &&\forall i,j=1,2\\
  1 &= \sum_{i=1}^{4} s_i s_i^* + 2 \sum_{i,j=1}^{2} t_{ij} t_{ij}^*
\end{align*}
where \(t = (t_{ij})_{i,j=1}^2 \in M_2(\mathcal{O}(G))\).

\begin{theorem}\label{thm:qCuntzKrieger}
  The quantum Cuntz-Krieger algebra \(\mathcal{O}(G)\) of \(G = \Cay(\KP, I_2)\) is isomorphic to the (classical) Cuntz-Krieger algebra of the complete bipartite graph \(K_{4,4}\). More explicitly it is isomorphic to \(M_5(\CC) \otimes \mathcal{O}_{16}\) where \(\mathcal{O}_{16}\) denotes the Cuntz algebra with 16 generators.
\end{theorem}
\begin{proof}
  The element \(q \coloneqq \sum_{i=1}^{4} s_i s_i^*\) is a projection in \(\mathcal{O}(G)\) since \(t\) is a partial isometry and \(t^*t = \operatorname{diag}(q,q)\) is a projection in \(M_2(\mathcal{O}(G))\). This shows in particular that the partial isometries \(s_i\) have orthogonal range projections. Their source projections are all the same because \(s_i^* s_i = 2 \sum_{i,j=1}^{2} t_{ij} t_{ij}^* = 1-q \eqqcolon p\), so we have the relations
  \begin{align}
    s_i^* s_j &= \delta_{ij} p,\label{eq:rel1}\\
    \sum_i s_i s_i^* &= q,\label{eq:rel2}\\
    2\sum_{i,j} t_{ij} t_{ij}^* &= p\label{eq:rel3}.
  \end{align}
  Compressing (\ref{eq:rel3}) with \(q\) and using \(t = t \operatorname{diag}(q,q)\) shows that \(q t_{ij} = 0\), thus \(t_{ij} \in p\mathcal{O}(G)q\).
  The row operator
  \[
    V \coloneqq [v_1, v_2, v_3, v_4] \coloneqq \sqrt{2} [t_{11}, t_{12}, t_{21}, t_{22}] \in M_{1,4}(p \mathcal{O}(G) q)
  \]
  is a partial isometry since \(VV^* = p\) and \(VV^*V\) = \(pV = V\), hence \(R \coloneqq V^*V \in M_4(q \mathcal{O}(G) q)\) is a projection. Its operator-valued trace is given by
  \[
    \sum_{k=1}^{4} P_{kk} = 2 \sum_{i,j=1}^{2} t_{ij}^* t_{ij} = 2 \sum_{j=1}^{2} [t^*t]_{jj} = 4q
  \]
  and since the operator-valued trace is faithful, we have \(P = \operatorname{diag}(q,q,q,q)\), hence
  \[
    v_k^* v_l = \delta_{kl} q.
  \]
  Together with (\ref{eq:rel1}), (\ref{eq:rel2}) and (\ref{eq:rel3}) (\ie \(\sum_k v_k v_k^* = p\)), the \(s_i\) and \(v_k\) fulfill exactly the Cuntz-Krieger relations for the adjacency matrix of the complete bipartite graph \(K_{4,4}\). Conversely, given any such Cuntz-Krieger family \(a_i\), \(b_k\), it is easy to check that the elements \(\tilde{s}_i \coloneqq a_i\), \((\tilde{t}_{11}, \tilde{t}_{12}, \tilde{t}_{21}, \tilde{t}_{22}) \coloneqq (b_1, b_2, b_3, b_4)/\sqrt{2}\) fulfill the relations from above. Using the machinery of classical Cuntz-Krieger algebras or graph \(C^*\)-algebras, one can now identify this algebra as \(M_5(\CC) \otimes \mathcal{O}_{16}\), for example by computing its K-theory and invoking the Kirchberg classification theorem. We follow an alternative, more down to earth approach and write down an explicit isomorphism between \(\mathcal{O}(G)\) and \(M_5(\CC) \otimes \mathcal{O}_{16}\):

  Put \(w_0 \coloneqq p\) and \(w_i \coloneqq s_i\) for \(1 \leq i \leq 4\). Then \(w_i^* w_j = \delta_{ij} p\) and \(\sum_{i=0}^{4} w_i w_i^* = 1\). The sixteen elements
  \[
    u_{ki} \coloneqq v_k s_i \in p \mathcal{O}(G) p
  \]
  satisfy
  \[
    u_{ki}^* u_{lj} = s_i^* v_k^* v_l s_j = \delta_{kl} s_i^* q s_j = \delta_{kl} s_i^* s_j = \delta_{kl} \delta_{ij} p
  \]
  and
  \[
    \sum_{k,i} u_{ki} u_{ki}^* = \sum_{k,i} v_k s_i s_i^* v_k^* = \sum_k v_k q v_k^* = \sum_k v_k v_k^* = p.
  \]
  So if we denote the canonical generators of the Cuntz algebra \(\mathcal{O}_{16}\) by \(z_{ki}\), then
  \begin{align*}
    \rho \colon \mathcal{O}_{16} &\to p\mathcal{O}(G)p\\
    z_{ki} &\mapsto u_{ki}\\
    1 &\mapsto p
  \end{align*}
  yields a well-defined unital \(*\)-homomorphism. Putting things together we obtain a well-defined unital \(*\)-homomorphism
  \begin{align*}
    \Psi \colon M_5(\CC) \otimes \mathcal{O}_{16} &\to \mathcal{O}(G)\\
    E_{ij} \otimes x &\mapsto w_i \rho(x) w_j^*.
  \end{align*}
  For the other direction we observe that the elements
  \[
    S_i \coloneqq E_{i0} \otimes 1, \qquad V_k \coloneqq \sum_{i=1}^{4} E_{0i} \otimes z_{ki}
  \]
  satisfy the relations for \(s_i, v_k \in \mathcal{O}(G)\) where \(E_{00} \otimes 1\) plays the role of \(p\) and \(\sum_{i=1}^{4} E_{ii} \otimes 1\) plays the role of \(q\). Thus we also get a unital \(*\)-homomorphism
  \begin{align*}
    \Phi \colon \mathcal{O}(G) &\to M_5(\CC) \otimes \mathcal{O}_{16}\\
    s_i &\mapsto S_i\\
    v_k &\mapsto V_k.
  \end{align*}
  We have
  \[
    (\Psi \Phi) (s_i) = \Psi (E_{i0} \otimes 1) = w_i p w_0^* = s_i
  \]
  and
  \[
    (\Psi \Phi) (v_k) = \sum_{i=1}^{4} \Psi(E_{0i} \otimes z_{ki}) = \sum_{i=1}^{4} w_0 u_{ki} w_i^* = \sum_{i=1}^{4} p v_k s_i s_i^* = p v_k q = v_k
  \]
  so \(\Psi \Phi = \id\). In the reverse direction we have
  \[
    (\Phi \Psi) ( E_{ij} \otimes 1) = \Phi ( w_i p w_j^*) = E_{i0} \otimes 1 \cdot E_{00} \otimes 1 \cdot E_{0j} \otimes 1 = E_{ij} \otimes 1
  \]
  and
  \[
    (\Phi \Psi) ( E_{00} \otimes z_{ki}) = \Phi(p u_{ki} p) = \Phi(u_{ki}) = \Phi(v_k s_i) = \sum_{j=1}^{4} E_{0j} \otimes z_{kj} \cdot E_{i0} \otimes 1 = E_{00} \otimes z_{ki}
  \]
  and since the elements \(E_{ij} \otimes 1\) and \(E_{00} \otimes z_{ki}\) generate \(M_5(\CC) \otimes \mathcal{O}_{16}\), we also have \(\Phi \Psi = \id\).
\end{proof}

The above result is probably not very surprising because \(\Cay(\KP, I_2)\) behaves very similar to \(K_{4,4}\). In fact, we can even show that they are quantum isomorphic (in the sense of \autocite[Definition~6.5]{BrannanEiflerVoigtWeber2022}).

\begin{proposition}\label{prop:qisoK44}
  The quantum graph \(\Cay(\KP, I_2)\) is quantum isomorphic to \(K_{4,4}\).
\end{proposition}
\begin{proof}
  By \autocite[Theorem~6.6]{BrannanEiflerVoigtWeber2022} it suffices to find a non-zero \(*\)-representation of the linking algebra \(\mathcal{O}(G^+(\Cay(\KP, I_2), K_{4,4}))\), \ie we need to find a unital \(*\)-homomorphism
  \[
    \beta \colon C(\KP) \to \CC^8 \otimes C
  \]
  for some \(C^*\)-algebra \(C\) satisfying
  \begin{equation}
    (\Tr \otimes \id) \beta(x) = H(x) 1 \label{eq:qIso_rel1}
  \end{equation}
  for all \(x \in C(\KP)\) and
  \begin{equation}
    (A_2 \otimes \id)\beta = \beta A_1 \label{eq:qIso_rel2}
  \end{equation}
  where \(A_1\) is the quantum adjacency matrix of \(\Cay(\KP, I_2)\) and \(A_2\) is the adjacency matrix of \(K_{4,4}\).

  In fact, we can take \(C = M_2(\CC)\) and construct \(\beta\) as follows. Choose a unitary error basis \((w_i)_{i=1}^4\) of \(M_2(\CC)\), for example the Pauli matrices
  \[
    w_1 = I_2, \ w_2 =
    \begin{pmatrix}
      1 & 0\\ 0 & -1
    \end{pmatrix}, \ w_3 =
    \begin{pmatrix}
      0 & 1\\ 1 & 0
    \end{pmatrix}, \ w_4 =
    \begin{pmatrix}
      0 & -i\\ i & 0
    \end{pmatrix},
  \]
  satisfying \(\sum_{i=1}^{4} w_i x w_i^* = 2\Tr(x)I_2\). Label the vertices of \(K_{4,4}\) as \(U = \{u_1, \ldots, u_4\}\) and \(V = \{v_1, \ldots, v_4\}\) and define
  \begin{align*}
    \beta \colon C(\KP) &\to \CC^8 \otimes M_2(\CC) &&\\
    e_i &\mapsto \delta_{u_i} \otimes I_2 &&\big(e_i \in \CC^4\big),\\
    x &\mapsto \sum_{j=1}^{4} \delta_{v_j} \otimes w_j x w_j^* &&(x \in M_2(\CC)).
  \end{align*}
  This is certainly a unital \(*\)-homomorphism and we have
  \[
    (\Tr \otimes \id) \beta(e_i) = \Tr(\delta_{u_i}) I_2 = I_2 = H(e_i) I_2
  \]
  and
  \[
    (\Tr \otimes \id) \beta(x) = \sum_{j=1}^{4} \Tr(\delta_{v_j})w_j x w_j^* = 2\Tr(x)I_2 = H(x) I_2
  \]
  for \(x \in M_2(\CC)\), proving (\ref{eq:qIso_rel1}). For (\ref{eq:qIso_rel2}) we have
  \[
    (A_2 \otimes \id) \beta(e_i) = A_2(\delta_{u_i}) \otimes I_2 = \sum_{j=1}^{4} \delta_{v_j} \otimes I_2
  \]
  and
  \[
    (A_2 \otimes \id) \beta(x) = \sum_{j=1}^{4} A_2(\delta_{v_j}) \otimes w_j x w_j^* = \left( \sum_{i=1}^{4} \delta_{u_i} \right) \otimes \left( \sum_{j=1}^{4} w_j x w_j^* \right) = \left( \sum_{i=1}^{4} \delta_{u_i} \right) \otimes 2\Tr(x)I_2.
  \]
  On the other hand, \(A_1\) satisfies \(A_1(x) = 2\Tr(x) 1_{\CC^4}\), hence
  \[
    \beta(A_1(x)) = \beta(2\Tr(x)1_{\CC^4}) = 2\Tr(x) \sum_{i=1}^{4} \delta_{u_i} \otimes I_2
  \]
  which agrees with the expression above, so also (\ref{eq:qIso_rel2}) holds.
\end{proof}

\printbibliography

\end{document}